\documentclass[12pt,a4paper]{article}

\usepackage{amsmath}
\usepackage{amssymb}
\usepackage{amsthm}

\usepackage[caption=false]{subfig}% Support for small, `sub' figures and tables

\usepackage{graphics}

\usepackage[numbers,sort&compress]{natbib}% Citation support using natbib.sty
\bibpunct[, ]{[}{]}{,}{n}{,}{,}% Citation support using natbib.sty
\makeatletter% @ becomes a letter
\def\NAT@def@citea{\def\@citea{\NAT@separator}}% Suppress spaces between citations using natbib.sty
\makeatother% @ becomes a symbol again

\theoremstyle{plain}% Theorem-like structures provided by amsthm.sty
\newtheorem{theorem}{Theorem}[section]
\newtheorem{lemma}[theorem]{Lemma}

\theoremstyle{definition}
\newtheorem{definition}[theorem]{Definition}
\newtheorem{example}[theorem]{Example}

\theoremstyle{remark}
\newtheorem{remark}{Remark}

\begin{document}

\title{Non-autonomous overdetermined problems for the normalized $p\,$-Laplacian}

\author{
{L.~Cadeddu\footnote{Department of Mathematics and Computer Science, University of Cagliari, via Ospedale 72, 09124 Cagliari, Italy}, A.~Greco\textsuperscript{*}\thanks{CONTACT A.~Greco. Email: greco@unica.it}\, and B.~Mebrate\footnote{Department of Mathematics,
Wollo University,
Dessie, Ethiopia}}
}

\maketitle

\begin{abstract}
We present existence and nonexistence results on the solution of an overdetermined problem for the normalized $p$-Laplacian in a bounded open set, with $p$ ranging from~$1$ to infinity. More precisely we consider a non-constant Neumann condition at the boundary. The definitions and statements needed to understand the main results are recalled in detail.
\end{abstract}


\section{Introduction}

Let $\Omega\subset\mathbb{R}^{n}$ be a bounded open set with a differentiable boundary~$\partial\Omega$, whose outer normal we denote by~$\nu$. Choose $\bar x \in \Omega$ and define
\begin{equation}\label{eq1.01}
R_1=\min_{x\in\partial\Omega}|x-\bar x|,
\quad
R_2=\max_{x\in\partial\Omega}|x-\bar x|.
\end{equation}

In this paper we study overdetermined problems ruled by the normalized $p$-Laplacian for $p \in [1,\infty]$: more precisely, we study the problem 
\begin{equation}\label{eq1}
\left\{\begin{array}{ll}
-\Delta_p^N u = 1 & \mbox{in $\Omega$;}\\
\noalign{\medskip}
u = 0,\;-\dfrac{\partial u}{\partial\nu}= q(|x-\bar x|) & \mbox{on $\partial\Omega$,}
\end{array}\right.
\end{equation}
where $q(r)$ is a real-valued function defined on $[R_1,R_2]$. For a smooth function~$u$ with nonvanishing gradient $Du$, the operator $\Delta_p^N u$ is given by
\begin{equation}\label{normalized}
\Delta_p^N u
=
\begin{cases}
\frac1p
\,
|Du|^{2 - p}
\,
\mathop{\rm div} \big(|Du|^{p - 2} \, Du \big),
&p \in [1, \infty);
\\
\noalign{\medskip}
|Du|^{-2}
\,
\langle D^{2}u \, Du,Du \rangle
&p = \infty.
\end{cases}
\end{equation}
The relationship between the case when $p$ is finite and the case $p = \infty$ is put into evidence by the equality
$$
\Delta_p^N u = \frac{\, p - 1 \,}p \, \Delta_\infty^N \, u
+
\frac1{\, p \,} \, \Delta_1^N u
\quad
\mbox{for $p \in [1,\infty)$}
$$%
(see \cite[(1.2)]{BK1} or \cite[(1.6)]{K1}): for a given function $u$ and at a fixed point $x$ such that $Du(x) \ne 0$, it follows that $\Delta_p^N u(x) \to \Delta_\infty u(x)$ when $p \to \infty$. The term ``normalized'' is used to make a distinction from the \textit{classical} $p$-Laplace operator $\Delta_{p\,} u$ given by
\begin{equation}\label{classical}
\Delta_{p\,} u =
\begin{cases}
\mathop{\rm div} \big(|Du|^{p - 2} \, Du \big),
&p \in [1, \infty);
\\
\noalign{\medskip}
\langle D^2 u \, Du,Du \rangle,
&p = \infty.
\end{cases}
\end{equation}
Solutions to~\eqref{eq1} are intended in the viscosity sense for $p \in [1,\infty]$: the definition of a viscosity solution is recalled in Section~\ref{vis}, together with the meaning of $\partial u / \partial \nu$ in~\eqref{eq1}. In the special case when $p = 1$, we also consider solutions which are smooth near the boundary (see below). Concerning the usual Laplace operator $\Delta = \Delta_2$, in the fundamental work~\cite{S1} Serrin showed, in particular, that if $\Omega$ is sufficiently smooth and there exists a solution $u\in C^2(\overline{\Omega})$ to the problem
\begin{equation}\label{Serrin_problem}
\left\{\begin{array}{ll}
-\Delta u = 1 & \mbox{in $\Omega$;}\\
\noalign{\medskip}
u = 0,\;-\dfrac{\partial u}{\partial\nu}=q & \mbox{on $\partial\Omega$,}
\end{array}\right.
\end{equation}
where $q$ is any constant, then $\Omega$ is a ball centered at some $x_1 \in \mathbb R^n$ and $u$ is given by $u(x)=\dfrac{\, r^{2}-|x-x_1|^{2}}{2n}$, where $r$ is the radius of~$\Omega$. Notice that problem~\eqref{Serrin_problem} is invariant under translations, hence the point $x_1$ is arbitrary, while the radius $r$ depends on the value of $q$.
Buttazzo and Kawohl in \cite{BK2} studied the corresponding overdetermined problem both for the infinity-Laplacian:
\begin{equation}\label{eq1.101}
\left\{\begin{array}{ll}
-\Delta_\infty \, u = 1 & \mbox{in $\Omega$;}\\
\noalign{\medskip}
u = 0,\; -\dfrac{\partial u}{\partial\nu}= q & \mbox{on $\partial\Omega$}
\end{array}\right.
\end{equation}
and the normalized one:
$$
\left\{\begin{array}{ll}
-\Delta_\infty^N \, u = 1 & \mbox{in $\Omega$;}\\
\noalign{\medskip}
u = 0,\; -\dfrac{\partial u}{\partial\nu}= q & \mbox{on $\partial\Omega$.}
\end{array}\right.
$$
Banerjee and Kawohl \cite{BK1}, instead, considered the corresponding problem for the normalized $p$-Laplacian with $p \in (1,\infty)$. They proved that if $u\in C(\overline{\Omega})$ is a viscosity solution to
$$
\left\{\begin{array}{ll}
-\Delta_p^Nu = 1 & \mbox{in $\Omega$;}\\
\noalign{\medskip}
u = 0,\; -\dfrac{\partial u}{\partial\nu}= q & \mbox{on $\partial\Omega$},
\end{array}\right.
$$%
then $\Omega$ is a ball. The result is also true if $p=1$ provided that $u$ is smooth near the boundary \cite[Remark~4.3]{K1}. However, for $p=\infty$ it is generally false \cite[Theorem~2]{BK2}. Problem \eqref{eq1.101}, where $\Omega$ contains the origin and $u$ satisfies a non-constant Neumann condition at the boundary $\partial\Omega$, given by
\begin{equation}\label{Neumann}
-\dfrac{\partial u}{\partial\nu}=q(|x|),
\end{equation}
has been studied in~\cite{G2}. The corresponding problem for the classical $p$-Laplacian with finite~$p > 1$, namely
$$
\left\{\begin{array}{ll}
-\Delta_{p\,} u = 1 & \mbox{in $\Omega$;}\\
\noalign{\medskip}
u = 0,\; -\dfrac{\partial u}{\partial\nu}= q(|x|) & \mbox{on $\partial\Omega$},
\end{array}\right.
$$%
was considered in \cite{G3,G1,GP}. The main results were basically focused on the geometry of~$\Omega$. This paper deals with a similar problem related to the normalized $p$-Laplacian, $p \in [1,\infty]$. Most of the notations we use are standard. By $c_p$ we denote the constant
\begin{equation}\label{c_p}
c_p =
\begin{cases}
\frac{p}{\, p + n - 2 \,}, &p \in [1, \infty);
\\
\noalign{\medskip}
1, &p = \infty.
\end{cases}
\end{equation}
Our first result is the following:
\begin{theorem}\label{main}
Let $p \in [1,\infty]$ and let $c_p$ be as above.
\begin{enumerate}
\item Suppose that the equation $q(r)- c_p \, r=0$ possesses a unique solution $R\in [R_1,R_2],$ and\/ $(q(r)-c_p \, r)(r-R)>0$ for all\/ $r \in [R_1,R_2] \setminus \{R\}$. Then problem \eqref{eq1} has a viscosity solution only in the special case when $R_1=R=R_2$ (i.e., $\Omega=B(\bar x,R)$). If, instead, $R_1<R_2,$ then problem~\eqref{eq1} has no viscosity solution.
\item Suppose that the function $\rho(r)=\frac{q(r)}r$ is strictly increasing. Then problem \eqref{eq1} has a viscosity solution only if\/ $\Omega$ is a ball centered at~$\bar x$.
\item If\/ $q$ is continuous, and if the equation $q(r)-c_p \, r=0$ does not possess any solution, then problem~\eqref{eq1} has no viscosity solution.
\end{enumerate}
\end{theorem}
The theorem is proved in Section~\ref{proof} by means of a comparison argument. The result shows that the behavior of the normalized $p$-Laplacian with respect to the overdetermined problem~\eqref{eq1} enjoys a continuity property at infinity: more precisely, since $c_\infty = 1 = \lim\limits_{p \to \infty} c_p$, the statement for $p = \infty$ is readily obtained from the case when $p$ is finite by just letting $p \to \infty$. By contrast, the classical (not normalized) $p$-Laplacian~\eqref{classical} exhibits a different behavior: indeed, a result similar to Theorem~\ref{main} valid for $p \in (1,\infty)$ has been proved in \cite[Corollary~1.2]{GP}. There, the ratio $q(r)/r^\frac1{p - 1}$ (which is obtained by letting $\varepsilon_0 = p - 1$) is required to be non-decreasing: such a ratio tends to $q(r)$ when $p \to \infty$. However, the corresponding result for the infinity-Laplacian which is found in \cite[Theorem~1.1]{G2} requires monotonicity of $q(r)/r^{1/3}$. When $q$ is constant, counterexamples are known: see \cite[p.~241]{BK2}.

Unlike \cite[Corollary~1.2]{GP}, our Theorem~\ref{main} also applies to the special case when $p = 1$. In such a case, as mentioned before, we focus on solutions~$u$ which are smooth near the boundary. To this purpose, we adopt the following definition:

\begin{definition}
Let\/ $\Omega \subset \mathbb R^n$ be a bounded open set of class $C^2$, and for $\varepsilon > 0$ define $\Omega_\varepsilon = \{\, x \in \Omega \mid \mathop{\rm dist}(x, \partial \Omega) < \varepsilon \,\}$. We say that a viscosity solution to~{\rm(\ref{eq1})} is a \textit{smooth solution near the boundary} if:
\begin{enumerate}
\item $u \in C^2(\overline \Omega_\varepsilon)$ for some $\varepsilon > 0$;
\item $Du \ne 0$ in $\overline \Omega_\varepsilon$;
\item the equation $-\Delta_p^N u = 1$ is satisfied in the classical sense in $\overline \Omega_\varepsilon$;
\item the boundary conditions in~\eqref{eq1} hold pointwise. 
\end{enumerate}
\end{definition}

\noindent When $p = 1$, a solution~$u$ which is smooth near the boundary satisfies the equation
\begin{equation}\label{degenerate}
(n - 1) \, |Du| \, H(x) = 1
\quad\mbox{in $\overline \Omega_\varepsilon$,}
\end{equation}
where $H(x)$ is the mean curvature of the level surface $u = {}$constant passing through the point~$x$ (see \cite[(14.102)]{GT} and \cite[Remark 4.3]{K1}). Note that if \eqref{eq1} has a solution~$u$ which is smooth near the boundary, then the surface $\partial \Omega$ must have a positive mean curvature $H(x)$ as a consequence of~\eqref{degenerate}. In such a case we may prescribe the Neumann condition by means of a function $q(|x-\bar x|, \, H(x))$ that not only depends on the distance from $x \in \partial \Omega$ to some fixed point $\bar x \in \Omega$, but is also allowed to depend on the mean curvature $H(x)$ of~$\partial \Omega$. More precisely, we consider the overdetermined problem
\begin{equation}\label{eq1H}
\left\{\begin{array}{ll}
-\Delta_1^N u = 1 & \mbox{in $\Omega$;}\\
\noalign{\medskip}
u = 0,\;-\dfrac{\partial u}{\partial\nu}= q(|x-\bar x|, \, H(x)) & \mbox{on $\partial\Omega$,}
\end{array}\right.
\end{equation}
where $q\colon [R_1,R_2]\times(0,\infty) \to (0,\infty)$ is a prescribed, positive function. In the case when $q(r,h)$ is independent of~$h$, problem~\eqref{eq1H} clearly reduces to~\eqref{eq1}. We have:

\begin{theorem}\label{p=1}
Let\/ $\Omega$ be a bounded open set of class~$C^2$. Choose $\bar x \in \Omega$ and define $R_1,R_2$ as in~{\rm(\ref{eq1.01})}. Consider a positive function $q(r,h)$ such that
\begin{enumerate}
\item $q(r,h)$ is monotone non-decreasing in $h$ for every $r \in [R_1, R_2]$;
\item the ratio $\frac{q(r, \, 1/r)}r$ is strictly increasing.
\end{enumerate}
If problem \eqref{eq1H} has a solution $u$ which is smooth near the boundary, then $R_1=R_2$ (i.e., $\Omega=B(\bar x,R_1)$).
\end{theorem}

\begin{example}
If the function $q$ has the special form $q(r,h) = r^\alpha \, h^\beta$, then $q(r, \, 1/r) \allowbreak = r^{\alpha-\beta}$ and the assumptions in the theorem are satisfied provided that $\alpha - 1 > \beta \ge 0$.
\end{example}

\begin{remark}\label{equivalence}
Several symmetry results were obtained in~\cite{G2} for overdetermined problems related to the equation $\Delta_\infty u = 0$. As mentioned in\/ \cite[Remark~2.2, p.~599]{AS}, ``there is no difference between the two resulting equations (in the viscosity sense) when the right-hand side $f \equiv 0$''. Hence all results in~\cite{G2} concerning solutions to $\Delta_\infty u = 0$ also hold for normalized infinity-harmonic functions, i.e., for solutions of\/ $\Delta_\infty^N u = 0$.
\end{remark}

The paper is organized as follows: in Section~\ref{vis} we define viscosity solutions, and in Section~\ref{preliminary} we recall some preliminary lemmas which will be used to prove our main results in Section~\ref{proof}.

\section{Viscosity Solutions}\label{vis}
Let $p \in [1,\infty]$. As usual (see for instance~\cite{BK1,KMP,MM1}), if $u\in C(\Omega)$ is twice differentiable at $x_0\in\Omega$ and if $Du(x_0) \ne 0$, we define the upper and lower normalized $p$-Laplacian of $u$ at $x_0$, respectively by $\Delta_p^+ \, u(x_0) = \Delta_p^- \, u(x_0) = \Delta_{p\,} u(x_0)$, where $\Delta_{p\,} u$ is given by~\eqref{normalized}. If, instead, $Du(x_0) = 0$, we denote by $\lambda_{\rm min} = \lambda_1 \le \ldots \le \lambda_n = \lambda_{\rm max}$ the eigenvalues of the Hessian matrix $D^2 u(x_0)$ and define
$$
\Delta_p^+ \, u(x_0) =
\begin{cases}
\frac{\, p - 1 \,}p \, \lambda_1
+
\frac1{\, p \,} \, \sum\limits_{i = 2}^n \lambda_i,
&p \in [1,2];
\\
\noalign{\medskip}
\frac{\, p - 1 \,}p \, \lambda_n
+
\frac1{\, p \,} \, \sum\limits_{i = 1}^{n - 1} \lambda_i,
&p \in (2,\infty);
\\
\noalign{\medskip}
\lambda_n, &p = \infty,
\end{cases}
$$%
and
$$
\Delta_p^- \, u(x_0) =
\begin{cases}
\frac{\, p - 1 \,}p \, \lambda_n
+
\frac1{\, p \,} \, \sum\limits_{i = 1}^{n - 1} \lambda_i,
&p \in [1,2];
\\
\noalign{\medskip}
\frac{\, p - 1 \,}p \, \lambda_1
+
\frac1{\, p \,} \, \sum\limits_{i = 2}^n \lambda_i,
&p \in (2,\infty);
\\
\noalign{\medskip}
\lambda_1, &p = \infty.
\end{cases}
$$%
In the case when $p < \infty$, the definitions above may equivalently be rewritten as follows (cf.~\cite[p.~177]{KMP}):
$$
\Delta_p^+ \, u(x_0) =
\begin{cases}
\frac{\, p - 2 \,}p \, \lambda_{\rm min}
+
\frac1{\, p \,} \, \Delta u,
&p \in [1,2];
\\
\noalign{\medskip}
\frac{\, p - 2 \,}p \, \lambda_{\rm max}
+
\frac1{\, p \,} \, \Delta u,
&p \in (2,\infty).
\end{cases}
$$%
and
$$
\Delta_p^- \, u(x_0) =
\begin{cases}
\frac{\, p - 2 \,}p \, \lambda_{\rm max}
+
\frac1{\, p \,} \, \Delta u,
&p \in [1,2];
\\
\noalign{\medskip}
\frac{\, p - 2 \,}p \, \lambda_{\rm min}
+
\frac1{\, p \,} \, \Delta u,
&p \in (2,\infty).
\end{cases}
$$%
For $f \colon \Omega \to \mathbb R$ we will give the definition of viscosity solution to the PDE
\begin{equation}\label{eq0.3}
-\Delta_p^N u(x)=f(x)\;\text{in}\;\Omega
\end{equation}
We denote by $USC(\Omega)$ and $LSC(\Omega)$, respectively, the spaces of \textit{upper semicontinuous} and \textit{lower semicontinuous} real-valued functions on~$\Omega$. Furthermore, for any $x_0 \in \Omega$ and $\varphi \in C^2$ in a neighborhood of~$x_0$ we write $u\prec_{x_0}\varphi$ (respectively, $u\succ_{x_0}\varphi$) if the difference $u - \varphi$ has a local maximum (minimum) at $x_0$. The notation extends in an obvious way to the case when $u$ is also defined at some $x_0 \in \partial \Omega$.

\medskip

\begin{definition}
~
\begin{enumerate}
\item $u\in USC(\Omega)$ is called a viscosity subsolution (or simply subsolution) of the PDE \eqref{eq0.3} in $\Omega$ if for every $x_0 \in \Omega$, and for every $\varphi\in C^2(\Omega)$ satisfying $u\prec_{x_0}\varphi$, we have
$$-\Delta_p^+\varphi(x_0)\leq f(x_0).$$
In this case we write $-\Delta_p^N u(x)\leq f(x)$ in $\Omega$.
\item $u\in LSC(\Omega)$ is called a viscosity supersolution (or simply supersolution) of the PDE \eqref{eq0.3} in $\Omega$ if for every $x_0 \in \Omega$, and for every $\varphi\in C^2(\Omega)$ satisfying $u\succ_{x_0}\varphi$, we have
$$-\Delta_p^-\varphi(x_0)\geq f(x_0).$$
In this case we write $-\Delta_p^N u(x)\geq f(x)$ in $\Omega.$
\item $u\in C(\Omega)$ is called viscosity solution (or simply solution) of the PDE \eqref{eq0.3} in $\Omega,$ if u is both a subsolution and a supersolution.
\end{enumerate}
\end{definition}
\noindent We now consider a boundary datum~$g(x)$ and define a viscosity solution of the Dirichlet problem 
\begin{equation}\label{eq0}
\left\{\begin{array}{rcll}
-\Delta_p^N u &=& f, & \mbox{in $\Omega$}\\
u &=& g,& \mbox{on $\partial\Omega$}
\end{array}\right.
\end{equation} as follows. We also give a meaning to the boundary condition~\eqref{Neumann} (see \cite[Remark~1.2]{BK1}).
\goodbreak\begin{definition}\label{viscosity}
~
\begin{enumerate}
\item $u\in USC(\overline\Omega)$ is a subsolution of \eqref{eq0} if $u$ is a subsolution of $-\Delta_p^N u = f$ in $\Omega$ and satisfies $u\leq g$ on $\partial\Omega.$
\item $u\in LSC(\overline\Omega)$ is a supersolution of \eqref{eq0} if $u$ is a supersolution of $-\Delta_p^N u = f$ in $\Omega$ and satisfies $u\geq g$ on $\partial\Omega.$
\item $u\in C(\overline\Omega)$ is a solution of \eqref{eq0} if $u$ is both a subsolution and supersolution of \eqref{eq0}.
\item A solution $u$ of~\eqref{eq0} satisfies the boundary condition~\eqref{Neumann} if for every $x_0 \in \partial \Omega$ and every $\varphi \in C^2$ in a neighborhood of~$x_0$ we have: if $u\prec_{x_0}\varphi$ then $-\frac{\partial \varphi}{\partial\nu} \ge q(|x_0|)$; if, instead, $u\succ_{x_0}\varphi$ then $-\frac{\partial \varphi}{\partial\nu} \le q(|x_0|)$.
\end{enumerate}
\end{definition}

\begin{remark}\label{negative}
(i)~A smooth function $u$ with $Du \ne 0$ in\/~$\Omega$ satisfying \eqref{eq0} in the classical sense is also a viscosity solution. (ii)~The normalized $p$-Laplacian is a nonlinear operator for $p \ne 2$. Nevertheless, if $u\in USC(\Omega)$ is a subsolution of~\eqref{eq0}, then $v = -u\in LSC(\Omega)$ and $v$ is a supersolution of the Dirichlet problem
 \begin{equation}\label{eq0.5}
\left\{\begin{array}{rcll}
-\Delta_p^N v &=& -f, & \mbox{in $\Omega$}\\
v &=& -g,& \mbox{on $\partial\Omega$}
\end{array}\right.
\end{equation}
Similarly, if $u\in LSC(\Omega)$ is a supersolution of~\eqref{eq0}, then $v = -u \in USC(\Omega)$ and $v$ is a subsolution of~\eqref{eq0.5}.
\end{remark}

\section{Well-posedness, comparison principle, radial solutions}\label{preliminary}

The proof of Theorem~\ref{main} is based on the comparison principle and the explicit expression of the radial solutions which are recalled in this section.

\begin{lemma}[Comparison principle]\label{lemma1}
Let $p \in [1,\infty]$, let\/ $\Omega\subset\mathbb R^n$ be a bounded (possibly disconnected) open set, and $f \in C(\Omega)$. We assume that $f \ne 0$ in $\Omega$ and does not change sign. Let $u,v\in C(\overline{\Omega})$ satisfy 
$$-\Delta_p^N u\leq f(x)
\quad
\mbox{and}
\quad
-\Delta_p^N v\geq f(x),
\quad
x \in \Omega.
$$ 
If $u\leq v$ on $\partial\Omega,$ then $u\leq v$ in $\Omega$. The result also holds if $p = \infty$ and $f \equiv 0$ in $\Omega$.
\end{lemma}

\begin{proof}
The first claim follows from \cite[Theorem~5]{KMP}, taking Remark~\ref{negative} $(ii)$ into account. The case when $p = \infty$ and $f \equiv 0$ follows from Jensen's fundamental result \cite[Theorem~3.11]{Jensen} by virtue of the equivalence between infinity-harmonicity and normalized infinity-harmonicity (Remark~\ref{equivalence}). It is also a special case of \cite[Theorem 2.5]{MM3}.
\end{proof}

Uniqueness for problem~\eqref{eq0} is a consequence of the comparison principle stated above. Recall that uniqueness lacks in the case when $p = 1$ and $f \equiv 0$, and a famous example was given in~\cite[Section~3.6]{SZ} (see Fig.~\ref{counterexample} below and~\cite[Fig.~2]{K1}).
\begin{figure}[h]
\centering
\subfloat[A least-gradient function.]{%
\resizebox*{5.1cm}{!}{\includegraphics{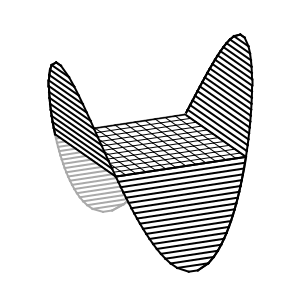}}}\hspace{20pt}
\subfloat[Another solution of $\Delta_1^N u = 0$.]{%
\resizebox*{5.1cm}{!}{\includegraphics{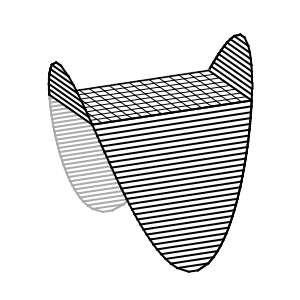}}}
\caption{Sternberg \& Ziemer's example.} \label{counterexample}
\end{figure}

\begin{lemma}[Radial solution]\label{lemma3}
Let $p \in [1,\infty]$ and let $c_p$ be as in~\eqref{c_p}. For every $R > 0$ the function
\begin{equation}\label{u_R}
u_R(x)=\dfrac{\, c_p \,}{2} \, \left(R^2-|x-\bar x|^2\right)
\end{equation}
is the unique solution of the problem
\begin{equation}\label{eq1.11}
\left\{\begin{array}{rcll}
-\Delta_p^N u &=& 1, & x\in B(\bar x,R)\\
\noalign{\medskip}
u &=& 0, & x\in \partial B(\bar x,R)
\end{array}\right.
\end{equation}
Moreover $-\dfrac{\partial u_R}{\partial\nu}=c_p \, R$, 
where $\nu = R^{-1} \, (x - \bar x)$ is the outer normal at $x \in \partial B(\bar x, R)$.
\end{lemma}
\begin{proof}
The solution is unique by Lemma~\ref{lemma1}. For $p \in (1,\infty)$, the representation~\eqref{u_R} is found in \cite[p.~20]{K1}. If $p = \infty$, the result follows from \cite[pag.~243]{BK2} with the observation that $d(x,\partial\Omega) = R - |x - \bar x|$ and by letting $a = R$. We give details for the case when $p = 1$, with reference to Definition~\ref{viscosity} and Remark~\ref{negative}~$(i)$. By differentiation of~\eqref{u_R}, $Du_R(x) = -\frac1{n - 1}(x-\bar x)$ and $D^2u_R(x) = -\frac1{n - 1} \, I$, where $I$ denotes the identity matrix. Hence $\lambda_{\rm min} = \lambda_{\rm max} = -\frac1{n-1}$. If $x \ne \bar x$, using~\eqref{normalized} we see that the equation is satisfied in the classical sense. Otherwise $Du_R(\bar x)=0$ and we have
$\Delta_1^+ \, u_R(\bar x) = \Delta_1^- \, u_R(\bar x) = \Delta u_R - \lambda_{\rm min} = -1$. Thus, $-\Delta_1^N \, u_R(x)=1$ in all of $B(\bar x,R)$ and, of course, $u_R(x)=0$ on $\partial B(\bar x,R)$.
\end{proof}

In order to prove Theorem~\ref{main}, we also need to establish the positivity of the solution to the following Dirichlet problem:
\begin{equation}\label{eq2}
\left\{\begin{array}{rcll}
-\Delta_p^N u &=& 1, & x\in\Omega\\
u &=& 0, & x\in \partial \Omega\\
\end{array}\right.
\end{equation}

\begin{lemma}\label{lemma4}
Let $p \in [1,\infty]$. Any solution of \eqref{eq2} is positive in\/~$\Omega$.
\end{lemma}
\begin{proof}
Notice that $v=0$ satisfies $-\Delta_p^N v\leq 1$ in $\Omega$ and $v=0$ on $\partial\Omega$. As $v$ is a subsolution, by the comparison principle (Lemma \ref{lemma1}) we have $0\leq u$ in $\Omega$. Consider $B(x_0,R)$ contained in $\Omega$, and consider the solution $u_R$ of problem \eqref{eq1.11} in $B(x_0,R)$. Since $u_R(x)=0$ on $\partial B(x_0,R)$ and $u\geq 0$ on $\partial B(x_0,R)$, Lemma \ref{lemma1} implies $0<u_R(x)\leq u(x)$ in $B(x_0,R)$. Since $x_0$ is arbitrary in $\Omega,$ we have $u>0$ in~$\Omega$.
\end{proof}

By choosing $B(x_0,R) \subset \Omega$ so that $\partial B(x_0,R) \cap \partial \Omega \ne \emptyset$, we immediately obtain the following boundary-point lemma (see also \cite[Lemma~2.3]{BK1}):

\begin{lemma}[Hopf]
Let $p \in [1,\infty]$. Suppose\/ $\Omega$ satisfies an interior sphere condition at every boundary point, and let $u\in C(\overline{\Omega})$ be a viscosity solution of \eqref{eq2}. Then for all $x\in\partial\Omega$ we have
$$\limsup_{t\rightarrow 0^{+}}\dfrac{u(x)-u(x-\nu t)}{t}<0.$$
\end{lemma}

We conclude this section by quoting some existence and regularity results.

\begin{lemma}[Existence]\label{lemma0}
Let\/ $\Omega\subset\mathbb{R}^{n}$ be a bounded (possibly disconnected) open set, and let $f\in C(\Omega)$ and $g\in C(\partial\Omega)$. The Dirichlet problem \eqref{eq0} has a viscosity solution provided that one of the following conditions hold:
\begin{enumerate}
\item $p \in (n,\infty]$, and\/ $f$ does not vanish and does not change sign in\/ $\Omega$.
\item $p = \infty$ and $f$ is bounded in\/ $\Omega$.
\end{enumerate}
\end{lemma}

\goodbreak
\begin{proof}
If $p = \infty$, Claim (1) follows from \cite[Theorem 1.8]{LW1} using Remark~\ref{negative} \textit{(ii)}. The claim was extended to $p > n$ in~\cite[Corollary 4.5]{K2}: indeed, assumption (3.19) of~\cite{K2} reduces to $p > n$. Claim (2) is a special case of \cite[Theorem 6.1]{MM1} corresponding to $F(x) = |x|$.
\end{proof}

\begin{remark}
Concerning regularity, global $C^{1,\beta}$-regularity is proved in \cite[Theorem~4.2]{BK1} for $p \in (1,\infty)$. In the case when $p = \infty$ it is known that the viscosity solution to~\eqref{eq2} is locally Lipschitz continuous in $\Omega$: see, for instance, \cite[Lemma~5.3]{WH1} with $F(x) = |x|$. See also \cite{CF} for further details.
\end{remark}

Let us point out that the existence, uniqueness and regularity results recalled above allow to construct the following counterexample, which mimics the one in \cite[p.~242]{G2}. The example is valid for $p \in (n,\infty)$ and shows that if we let the function $q$ in~\eqref{eq1} be arbitrary, i.e., if we drop every assumption on~$q$, then problem~\eqref{eq1} may well be solvable even though the domain~$\Omega$ is not a ball.

\begin{example}
Let\/ $\Omega \subset \mathbb R^2$ be an ellipse in canonical position, with semi-axes $a < b$, and let $p \in (n,\infty)$. Thus, there exists a unique solution $u = u_0$ of the Dirichlet problem~\eqref{eq2}. Note that the problem is invariant under reflection with respect to each axis: i.e., if we define $v(x_1,x_2) = u_0(\pm x_1, \pm x_2)$ for whatever choice of the signs $\pm$, we always find $\Delta_p^N v = \Delta_p^N u_0$. But since the solution of problem~\eqref{eq2} is unique, we must have $u_0 \equiv v$, hence $u_0(x_1,x_2) = u_0(\pm x_1, \pm x_2)$. Furthermore, since $u_0$ is differentiable up to the boundary, the last equality implies that\/ $|\nabla u_0(x_1,x_2)| = |\nabla u_0(\pm x_1, \pm x_2)|$ for every $(x_1,x_2) \in \partial \Omega$. Now observe that for every $r \in [a,b]$ the set $F_r$ of all $x \in \partial \Omega$ such that $|x| = r$ is invariant under reflection with respect to each axis, and therefore it is legitimate to define $q(r) = |\nabla u_0(x)|$ by choosing any $x = (x_1,x_2) \in F_r$ (because the value of $q(r)$ is independent of the choice of $x \in F_r$). Then, with this particular function $q$, problem~\eqref{eq2} is solvable (and has the solution~$u_0$) although\/ $\Omega$ is not a disc.
\end{example}

\section{Existence and nonexistence of solutions}\label{proof}
In this section we prove our main results.

\begin{proof}[Proof of Theorem~\ref{main}]
We follow the same guidelines as in \cite{G2}.
(1)\;Let $\Omega=B(\bar x,R)$, where $R$ is the solution of $q(r)- c_p \, r=0$. By Lemma~\ref{lemma3}, $u_R(x)=\frac{\, c_p \,}{2} \, (R^2-|x-\bar x|^2)$ is the solution to the problem \eqref{eq1} in $B(\bar x,R)$. Therefore, the solution to \eqref{eq1} exists in $B(\bar x,R)$. \\
On the other hand, assume $u$ is the solution to \eqref{eq1}. Define $u_{i}(x)=\frac{\, c_p \,}{2} \, ( R_i^2 - |x-\bar x|^2)$ for $i=1,2$. Then $u_{i}$ is the solution to \eqref{eq2} in the ball $B(\bar x,R_{i})$. Since $u\geq 0$ on $\partial B(\bar x,R_1)$ (see Lemma~\ref{lemma4}) and $u_1=0$ on $\partial B(\bar x,R_1),$ we have $u_1\leq u$ on $\partial B(\bar x,R_1)$. By Lemma~\ref{lemma1}, $u_1\leq u$ in $B(\bar x,R_1)$. 
Since $u_2\geq 0$ on $\partial\Omega$ and $u=0$ on $\partial\Omega,$ we have $u=0\leq u_2$ on $\partial\Omega$ and hence $u\leq u_2$ in $\Omega$ by Lemma~\ref{lemma1}. Let $P_1\in\partial B(\bar x,R_1)\cap\partial\Omega$. Then the outer normal $\nu$ to $\partial\Omega$ at $P_1$ equals $\dfrac{P_1-\bar x}{R_1}$, the outer normal to $B(\bar x,R_1)$. Since $u_1$ is a smooth function satisfying $u \succ_{P_1} u_1$, by Definition~\ref{viscosity}~(4) we may write
\begin{equation}\label{eq3}
c_p \, R_1=-\dfrac{\partial u_1}{\partial\nu}(P_1)\leq q(R_1).
\end{equation} 
Let $P_2\in\partial B(\bar x,R_2)\cap\partial\Omega$. Then the outer normal $\nu$ to $\partial\Omega$ at $P_2$ equals $\dfrac{P_2-\bar x}{R_2}$, the outer normal to $B(\bar x,R_2)$. Furthermore $u \prec_{P_2} u_2$. Hence 
\begin{equation}\label{eq4}
q(R_2) \leq -\dfrac{\partial u_2}{\partial\nu}(P_2)= c_p \, R_2.
\end{equation} 
Inequalities \eqref{eq3} and~\eqref{eq4} may be rephrased as
\begin{equation}\label{inequalities}
\mbox{$q(R_1) - c_p \, R_1\geq 0$ and $q(R_2) - c_p \, R_2\leq 0$.}
\end{equation}
Since the equation $q(r) - c_p \, r=0$ has the unique solution $R$ in $[R_1,R_2]$ and $q(r) - c_p \, r<0$ for $r<R,$ we have $R=R_1$ and again since $q(r) - c_p \, r>0$ for $r>R,$ we have $R=R_2$. Therefore, $R_1=R=R_2,$ which is $\Omega=B(\bar x,R).$

(2)\; If \eqref{eq1} has a solution, then we obtain \eqref{eq3} and \eqref{eq4}, hence $\rho(r)$ satisfies
$$\rho(R_2)=\dfrac{q(R_2)}{R_2}\leq c_p\leq \dfrac{q(R_1)}{R_1}=\rho(R_1).$$
Since $\rho$ is strictly increasing,  we must have $R_1=R_2$ and the result follows.

(3)\; Suppose that \eqref{eq1} has a solution. Then, by~\eqref{inequalities}, and since $q$ is continuous, we have $q(R) - c_p \, R=0$ at some point $R\in [R_1,R_2]$, contradicting the assumption. Therefore problem~\eqref{eq1} must be unsolvable.
\end{proof}

\begin{proof}[Proof of Theorem~\ref{p=1}]
The result follows by exploiting~\eqref{degenerate}. As mentioned in the Introduction, we have $H(x) > 0$ for every $x \in \partial \Omega$.
Take $P_i \in \partial B(\bar x, R_i) \cap \partial \Omega$, $i = 1,2$ as in the proof of Theorem~\ref{main}, and recall that the mean curvature of the sphere $\partial B(\bar x, R_i)$ is $1/R_i$. Hence we may write $H(P_1) \le 1/R_1$ and $H(P_2) \ge 1/R_2$. This and~\eqref{degenerate} imply
\begin{align*}
q(R_1, \, 1/R_1) \ge q(R_1,H(P_1)) = |Du(P_1)| = \frac1{\, (n - 1) \, H(P_1) \,}
\ge \frac{R_1}{\, n - 1 \,} = c_1 \, R_1
\\
q(R_2, \, 1/R_2) \le q(R_2,H(P_2)) = |Du(P_2)| = \frac1{\, (n - 1) \, H(P_2) \,}
\le \frac{R_2}{\, n - 1 \,} = c_1 \, R_2
\end{align*}
because $q(r,h)$ is monotone non-decreasing in~$h$.
The inequalities above imply
$$
\frac{\, q(R_1, \, 1/R_1) \,}{R_1}
\ge
c_1
\ge
\frac{\, q(R_2, \, 1/R_2) \,}{R_2}
.
$$%
Then, since the ratio $\frac{q(r, \, 1/r)}r$ is strictly increasing in~$r$, we must have $R_1 = R_2$ as claimed.
\end{proof}

\section*{Funding}

The authors are partially supported by the research project \textit{Analysis of PDEs in connection with real phenomena}, CUP F73C22001130007, funded by Fondazione di Sardegna (annuity 2021). L.~Cadeddu and A.~Greco are members of the Gruppo Nazionale per l'A\-na\-li\-si Matematica, la Probabilit\`a e le loro Applicazioni (GNAMPA) of the Istituto Na\-zio\-na\-le di Alta Matematica (INdAM). L.~Cadeddu's research has been accomplished within the UMI Group TAA {\it Approximation Theory and Applications}.

\end{document}